\documentclass{amsart}
\usepackage{amsmath}
\usepackage{amssymb}
\usepackage{amsthm,amssymb,amsmath,enumerate,graphicx}
\usepackage{enumitem}
\usepackage{enumerate}

\newtheorem{theorem}{Theorem}[section]
\newtheorem{lemma}[theorem]{Lemma}
\newtheorem{assertion}[theorem]{Assertion}

\newtheorem{observation}[theorem]{Observation}
\newtheorem{corollary}[theorem]{Corollary}

\newtheorem{conjecture}[theorem]{Conjecture}

\theoremstyle{remark}

\newtheorem{remark}[theorem]{Remark}
\newtheorem{definition}[theorem]{Definition}
\newtheorem{example}[theorem]{Example}
\newcommand{\orkef}{OR(K \cup \{e\}, F)}

\newcommand{\cc}{\mathcal C}

\newcommand{\cm}{\mathcal M}

\newcommand{\cp}{\mathcal P}

\begin{document}

\makeatletter

\makeatother
\author{Ron Aharoni}
\address{Department of Mathematics\\ Technion}
\email[Ron Aharoni]{raharoni@gmail.com}
\thanks{\noindent The research of the first author was
supported by BSF grant no. $2006099$,  by the Technion's research promotion fund, and by
the Discont Bank chair.}
\author{Eli Berger}
\address{Department of Mathematics\\ University of Haifa}
\email[Eli Berger]{eberger@haifa}
\thanks{\noindent The research of the second author was
supported by BSF grant no. $2006099$ and by ISF grant no.}

\author{Maria Chudnovsky}\thanks{The research of the third  author was
supported by  BSF grant no.
2006099, and NSF grants DMS-1001091 and IIS-1117631.}
\address{Department of Mathematics, Princeton University
}
\email[Maria Chudnovsky]{ mchudnov@math.princeton.edu}

\author{David Howard}
\address{Department of Mathematics\\ Colgate University}
\email[David Howard]{dmhoward@colgate.edu}
\thanks{\noindent The research of the fourth author was
supported by BSF grant no. $2006099$, and by ISF grants Nos.
$779/08$, $859/08$ and $938/06$.
 }

\author{Paul Seymour}
\address{Department of Mathematics\\ Princeton University}
\email[Paul Seymour]{pds@math.princeton.edu}
\thanks{\noindent The research of the fifth author was
supported by
 }

\title{Large rainbow matchings in general graphs}

\begin{abstract}
By a theorem of Drisko, any $2n-1$ matchings of size $n$ in a bipartite graph have a   rainbow matching of size $n$.  Inspired by remarks of Bar\'at,   Gy\'arf\'as and S\'ark\"ozy, we conjecture that if $n$ is odd then the same is true
 also in general graphs, and that if $n$ is even then $2n$ matchings of size $n$ suffice. We prove that any $3n-2$ matchings of size $n$  have a   rainbow matching of size $n$.
 \end{abstract}

\maketitle
\section{Introduction}

Given a system $\cc=(C_1, \ldots ,C_m)$
 of sets of edges in a graph, a {\em   rainbow matching} for $\cc$ is a matching, each of whose edges is chosen from a different $C_i$. Note that we do not insist that the rainbow matching represents all $C_i$s, so in fact our rainbow matchings are usually only partial.
 In this paper we   consider the case in which the sets $C_i$ are themselves matchings, and we are interested in questions of the form ``how many matchings of  size $k$ are needed to guarantee the existence of a   rainbow matching of size  $m$.'' It is conjectured \cite{ab} that $n$ matchings of size $n+1$ in a bipartite graph have a rainbow matching of size $n$, and the authors are not aware of any  example refuting the possibility that $n$ matchings of size $n+2$ in a general graph have a rainbow matching of size $n$. In  the bipartite case the best current results are  that  $n$ matchings of size $\lceil \frac{3}{2}n\rceil$ have a   rainbow matching of size $n$ \cite{akz}, and that
 $n$ matchings of size  $n+o(n)$ have a   rainbow matching of size $n$ \cite{pok}.

 A surprising  jump occurs when
 we insist that the matchings are of size $n$:  we need to take $2n-1$ such matchings  in a bipartite graph to guarantee a   rainbow matching of size $n$. The following example shows that $2n-1$ is best possible:

 \begin{example} \label{example} Let
 $M_i, 1 \le i \le n-1$ to be all equal to one of the two
perfect matchings in the cycle $C_{2n}$ and $M_i, ~~i \le 2n-2$ to be all equal
to the other perfect matching. This is a system of $2n-2$ matchings of size $n$ that does not have a   rainbow matching of size $n$.
\end{example}

  The fact that $2n-1$ matchings suffice is essentially due to Drisko  \cite{drisko}, who proved the following special case:
\begin{theorem}\label{originaldrisko}
Let $A$ be an $m\times n$ matrix in which the entries of each row are all distinct.
If $m\ge 2n-1$, then $A$ has a transversal, namely  a set of n distinct entries with no two in the same row or column.
\end{theorem}

In \cite{ahj} it was observed that this theorem follows also from  B\'arany's colorful extension of Caratheodory's theorem, and the fact that in bipartite graphs $\nu=\nu^*$, meaning  that the notion of ``matching''  can be replaced by that of ``fractional matching''.

In \cite{ab} the theorem  was formulated in the rainbow matchings  setting, and
given a short proof.

\begin{theorem}\label{drisko}
Any family $\cm=(M_1, \ldots, M_{2n-1})$  of matchings of size $n$ in a bipartite graph possesses a rainbow matching.
\end{theorem}

In \cite{akz} it was shown that Example \ref{example} is the only
instance in which $2n-2$ matchings do not suffice.
In \cite{akzdrisko} Theorem \ref{drisko} was strengthened, using topological methods:

\begin{theorem}
If $M_i,~~i=1,\ldots ,2n-1$ are matchings in a bipartite graphs satisfying $|M_i|=\min(i,n)$ for all $i \le 2n-1$, then there exists a rainbow matching of size $n$.
\end{theorem}

 Bar\'at,   Gy\'arf\'as and S\'ark\"ozy considered the same problem in general graphs, and from their comments  the following conjecture suggests itself:

\begin{conjecture}\label{bgs}\cite{bgs}
For $n$ even any $2n$ matchings of size $n$ in any graph have a   rainbow matching of size $n$, and for $n$ odd any $2n-1$ matchings  of size $n$ have a rainbow matching of size $n$.
\end{conjecture}

In \cite{ahj} a fractional version of this conjecture was proved, in an even stronger version, in which the given matchings are replaced by fractional matchings:

\begin{theorem}
Let $F_1, \ldots ,F_{2n}$ be sets of edges in a general graph, satisfying $\nu^*(F_i)\ge n$. Then there exists a rainbow set $f_1 \in F_1, \ldots ,f_{2n} \in F_{2n}$ such that $\nu^*(\{f_1, \ldots ,f_{2n}\})\ge n$.
\end{theorem}

This is generalized in \cite{ahj} to $r$-uniform $r$-partite hypergraphs, as follows: for any integer $r$ and any real number $m$, any  $\lceil rm-r+1 \rceil$ fractional matchings of size $m$ in an $r$-partite hypergraph have a rainbow fractional matching of size $m$.

There is a family of examples showing that  for $n$ even $2n-1$ matchings of size $n$  do not necessarily have a rainbow matching of size $n$. Its construction is based on the following observation:

\begin{observation}
  Let  $e=v_iv_j$  be an edge connecting two vertices of $C_{2k}=v_1v_2\ldots v_{2k}$. Then $e$ belongs to a perfect matching contained in $E(C_{2k}) \cup \{e\}$ if and only if $j-i$ (the length of $e$) is odd.
\end{observation}

The necessity of the condition follows from the fact that  in order for $e$ to participate in a perfect matching it has to
enclose an even number of vertices on each of its sides,  since the enclosed vertices have to be matched within themselves.  The sufficiency follows from the fact that if the number of enclosed vertices is even, they can be matched within the cycle.

\begin{example}\label{basicexample}
To the system of matchings $M_i$ in Example \ref{example} add a matching
 $C$ of size $n$, all of whose edges are of even length.
 The obtained family,  consisting of $2n-1$ matchings, does not possess a rainbow matching of size $n$, since an even length edge cannot be completed by edges from the initial system of matchings to a perfect matching.

\end{example}
Note that a  matching $C$ as above   exists if and only if $n$ is even.
The necessity of the evenness  condition follows from the fact that for $C$ as above,
$\sum \{i+j \mid v_iv_j \in C\}$ is even, and on the other hand it is equal to $\sum \{i  \mid 0 \le i \le 2n-1\}=  \binom{2n}{2}$. The sufficiency is shown by a simple construction.

The aim of this paper is to prove a weaker version of Conjecture \ref{bgs}:

\begin{theorem}\label{main}
$3n-2$ matchings of size $n$ in any graph have a   rainbow matching of size $n$.
\end{theorem}

\section{Preliminaries and notation}
We shall use the following  notation concerning paths. The first vertex on a path $P$ is denoted by $in(P)$, and its last vertex by $ter(P)$. The edge set of $P$ is denoted by $E(P)$, and its vertex set by $V(P)$.  For a path $P$ and a vertex $v$ on it, we denote by $Pv$ the part of $P$ between $in(P)$ and $v$, including $in(P)$ and $v$,  and by $vP$ the part of $P$ from $v$ to $ter(P)$,  including $v$ and $ter(P)$. If $P,Q$ are paths such that $in(Q)=ter(P)=v$ we shall use the notations $P*Q$,  or  $PvQ$, or also $PQ$, for the trail (namely a path that is not necessarily simple) resulting from the concatenation of $P$ and $Q$. Our paths are usually undirected, but sometimes we shall take them in one of their two possible directions. In such a case, we denote by $\overleftarrow{P}$ the path $P$ traversed in the opposite direction to that of $P$.

For a set $J$ of edges $\bigcup J$ is the set of vertices participating in $J$.

Let $F$ be a matching in a graph, and let $K$ be a set of edges disjoint from $F$. A path $P$ is said to be $K-F$-{\em alternating} if every odd-numbered edge of $P$ belongs to $K$ and every even-numbered edge belongs to $F$.  If there is no restriction on
the odd edges of $P$ then we just say that it is $F$-alternating. If $in(P) \not \in  \bigcup F$ we say that $P$ is {\em free-starting} and if $ter(P) \not \in  \bigcup F$ we say that $P$ is {\em free-ending} (the identity of $F$ is suppressed here). An alternating path that is both free-starting and free-ending is said to be {\em augmenting}. The origin of the name is that if $P$ is augmenting then $E(P) \bigtriangleup F$ is a matching larger than $F$. The converse is also  true:

\begin{lemma}\label{aug}
If $F,G$ are  matchings and $|G|>|F|$ then $E(F) \cup E(G)$ contains an $F$-alternating augmenting path.
\end{lemma}
 \begin{proof}
 Viewed as a multigraph, the connected components of $E(F) \cup E(G)$ are cycles (possibly digons) and paths that alternate between $G$ and $F$ edges. Since $|G|>|F|$ one of these paths contains more edges from $G$ than from $F$, and is thus $F$-augmenting.
 \end{proof}

\begin{definition}
Let $F$ be a matching, let $K$ be a set of edges disjoint from $F$, and let $a$ be any vertex.
A vertex $v$ is said to be {\em oddly $K$-reachable} (resp. {\em evenly $K$-reachable}) from $a$ if there
exists an odd (respectively even)  $K-F$-alternating path starting with an edge $ab \in K$ and ending at $v$. Being an odd alternating path means ending with an edge from $K$, and being an even alternating path means ending with an edge of $F$.
Let $OR(a,K,F)$ be the set of vertices oddly reachable  from $a$, $ER(a,K,F)$ the set of vertices evenly reachable  from $a$, and let $DR(a,K,F)= OR(a,K,F) \cap ER(a,K,F)$.  Let $OR(K,F)$  (respectively $ER(K,F)$) be the set of oddly $K$-reachable (respectively evenly $K$-reachable) vertices from some vertex not belonging to $\bigcup F$.
\end{definition}

Note that $V(G) \setminus \bigcup F \subseteq ER(K,F)$, since a vertex not matched by $F$ has a zero length alternating path to itself.  Note also that there exists a $K-F$ augmenting alternating path if and only if $OR(K,F) \not \subseteq \bigcup F$.

\begin{definition}
A graph $G$ is called {\em hypomatchable} if $G-v$ has a perfect matching for every $v \in V(G)$.

\end{definition}
Another term used for such graphs is {\em factor-critical.}



\begin{lemma}\label{hypovertex}
Let $F$ be a matching in a graph $G$, let $K =E(G)\setminus F$, and suppose that  $V(G) \setminus \bigcup F$ consists of a single vertex $a$. Then a vertex $x$ belongs to $ER(a,K,F)$ if and only if $G - x$ has a perfect matching.
\end{lemma}

\begin{proof}
Suppose that there exists a perfect matching $M$ of $G-x$. Then the $F-M$-alternating path
starting at $x$ with an edge of $F$ must terminate at $a$ with an edge of $M$. Reversing this path, we obtain a $K-F$ alternating path, starting at $a$ and reaching $x$, with last edge belonging to $F$. This shows that $x  \in ER(a,K,F)$. For the other direction, if  $x  \in ER(a,K,F)$ then taking $L$ to be the even (namely ending with an $F$-edge) $a-x$ $F$-alternating path reaching $x$ and letting $M=F \triangle L$ yields a perfect matching of $G-x$.

\end{proof}

For a graph (which for this purpose is a set of edges) $F$ and a vertex set $S$ we write $F[S]=\{f \in F \mid f \subseteq S\}$.
Given a matching $J$ and a vertex $v \in \bigcup J$, we denote by $J(v)$ the vertex $u$ for which $uv \in J$. Clearly, $x  \in OR(a,K,F)$ if and only if $F(x)$  belongs to $ER(a,K,F)$. Hence Lemma \ref{hypovertex} implies:

\begin{corollary}\label{doublereachability}
Let $F$ be a matching in a graph $G$, covering all vertices apart from one vertex, $a$. Let $K=E(G) \setminus F$. Then $G$ is hypomatchable if and only if
$V(G)=DR(a,K,F)$.
\end{corollary}

Combining the corollary with Lemma \ref{hypovertex} gives:

\begin{lemma}\label{hypo}
Let $H$ be a graph and $J$ a matching in it that covers all vertices except for one vertex, $a$. Let $K=E(H)\setminus J$. Then the following are equivalent:
\begin{enumerate}

\item
$V(H)=DR(a,K,J)$.

\item $V(H)=ER(a,K,J)$.

\item
$H$ is hypomatchable.
\end{enumerate}
\end{lemma}

\section{The main lemma}
Throughout this section and the next,
 $F$ is a maximum matching in a graph $G$, and $K=E(G) \setminus F$.
 An edge $e$ contained in $V(G)$ is called {\em enriching} if $OR(K \cup \{e\},F) \supsetneqq OR(K,F)$. In particular, an edge whose addition
to $E(G)$ generates a matching larger than $F$ is enriching. Note that if $e$ is enriching then $e \not \in E(G)$.

The key step in the proof of Theorem \ref{main} is:

\begin{lemma}\label{mainlemma}
\label{mainreach}
Any augmenting $F$-alternating path  contains an enriching edge.
\end{lemma}

For clarity, let us remind that an augmenting path is not contained in $G$.

To prove the lemma, we shall use the Gallai-Edmonds decomposition theorem.
There are many formulations of this theorem, of which we shall choose the following.

\begin{theorem}\label{gallaiedmonds}  \cite{edmonds,gallai, lp, goemans}

Let $F$ be a maximum matching in a graph $G$.  Then the vertex set of  $G$ can be decomposed into  disjoint sets  $Q,R,S$, so that the following hold:

\begin{enumerate}
\item
 $F[Q]$ is a perfect matching in $Q$.

\item
$R$ is the union of sets $V(H_i), ~i \in I$, where each $H_i$ is a hypomatchable connected component of $G-S$, and $F[V(H_i)]$ matches all of $V(H_i)$ except for one vertex $r_i$.
\item
$S \subseteq \bigcup F$, and for every $s \in S$  we have $F(s) =r_i$ for some $i \in I$.

\end{enumerate}

\end{theorem}

\begin{remark}
A more common formulation of the theorem is ``There exists a maximum matching for which there is a decomposition...'' The stronger version appearing here, ``For every maximum matching there exists a decomposition...'' follows easily from the seemingly weaker one.
\end{remark}

Let  $X=\bigcup F$  and $Y=V(G) \setminus X$.  Let $J =\{i \in I \mid r_i \in X\}$,  and $D=I \setminus J$.  For every $j \in J$ let $s_j=F(r_j)$.


Clearly,
\begin{equation}\label{y}
Y =\{r_d \mid d \in D \}.
\end{equation}

Next we wish to characterize the vertices of  $OR(K,F)$. For this purpose we study  how does a free-starting $K-F$ alternating path look like. To avoid confusion, note that we are not speaking here about  $F$-alternating paths as in Lemma \ref{mainlemma},  that are not contained in $K \cup F$.

Let  $B$ be a free starting $K-F$-alternating path.
 By \eqref{y},  $in(B) = r_d$ for some $d \in D$.

\begin{assertion}
 $V(B) \cap  V(H_i) =\emptyset$ for every $i \in D \setminus \{d\}$.
\end{assertion}

 \begin{proof}
Assume that $B$ meets $V(H_i)$ for some $i\neq d$ in $D$.     Let $t$ be the first vertex of $B$ belonging to $V(H_i)$, and let $P$ be the path witnessing the fact that $t \in ER(r_d,K,F)$ (see Lemma \ref{hypovertex}). Then the path $Bt\overleftarrow{P}$ ending at $r_i$
is augmenting, contradicting the maximality of $F$.
\end{proof}

By the assertion and the fact that vertices in each $H_i$ are connected  only to vertices in $V(H_i) \cup S$, we must have

$$B=r_ds_{j_1}r_{j_1}P_1s_{j_2}r_{j_2}P_2\ldots $$

where $d \in D$ and $P_k$ is a path inside $H_{j_k}$. In words, the path starts at a vertex $r_d$ for some $d \in D$; its first edge is $r_ds_{j_1}$ for some $j_1 \in J$; its second edge is $s_{j_1}r_{j_1}$; it then roams $V(H_{j_1})$; when leaving $H_{j_1}$ it goes to $s_{j_2}$, then it roams $H_{j_2}$ and so forth.

By the above,   if $B$ meets $V(H_j)$ for some $j \in J$, then  $B$ must contain the edge $s_jr_j$, traversed in this direction.

This analysis, together with Corollary \ref{doublereachability} show that:

\begin{assertion}\label{er}
$ER(K,F) \subseteq \bigcup_{i \in I} V(H_i)$.
\end{assertion}

and:

\begin{assertion}\label{orcontained}
$OR(K,F) \subseteq S \cup \bigcup_{i \in I} (V(H_i)\setminus \{r_i\})$.
\end{assertion}

Our next step is  to show that it is possible to choose the Gallai-Edmonds decomposition in such a way that equality holds in the last two  assertions.

\begin{lemma}\label{betterge}
If $Q$ is containment-wise maximal
among all Gallai-Edmonds decompositions $(Q,R,S)$, then

\begin{equation}\label{eveneq}
ER(K,F) = \bigcup_{i \in I} V(H_i),
\end{equation}

and

\begin{equation}\label{oddeq}
OR(K,F) = S \cup \bigcup_{i \in I} (V(H_i)\setminus \{r_i\}).\end{equation}
\end{lemma}

In particular,
\eqref{oddeq} will imply:

\begin{equation}\label{333}
S \subseteq OR(K,F).
\end{equation}

\begin{proof}
We shall only prove \eqref{oddeq} - the proof of \eqref{eveneq} is the same. Claim \eqref{oddeq} is easily seen to be equivalent  to:

\begin{equation}\label{33}
V \setminus OR(K,F)=Q \cup \{r_i \mid i \in I\}.
\end{equation}

Form a bipartite graph $\Gamma$ whose one side is $I$ and the other side is $S$, where a pair $is~ (i \in I, s \in S)$ belongs  to $E(\Gamma)$  if some vertex in $H_i$ is connected to $s$. Let $Z$ be the matching $\{is \mid i \in I, s \in S, F(s) =r_i\}$. Let $I'$ be  the set of all $i \in I$ that are reachable by a $Z$-alternating path starting in  $V(\Gamma) \setminus \bigcup Z$, and $S'$ be the set of vertices of $S$ that are reachable by a $Z$-alternating path starting in $V(\Gamma) \setminus \bigcup Z$.  Let $I''=I \setminus I'$ and $S''=S \setminus S'$. Observe that $I'' \subseteq J$, and so $Z(i)$ exists for every $i \in I''$.

First we show that $S' \cup \bigcup_{i \in I'}V(H_i) \subseteq OR(K,F)$. Let $u \in S' \cup \bigcup_{i \in I'}V(H_i)$.
Let $\tilde{u}=u$ if $u \in S'$, and let $\tilde{u}=i$ if $u \in V(H_i)$. Then there is a $Z$-alternating path
leading to $\tilde{u}$ starting outside of $\bigcup Z$. Let $P$ be such a path. Namely, the vertices of $P$ are
$p_1 \ldots p_k$ in order, $p_1 \not \in \bigcup Z$, and $p_k=\tilde{u}$. Since $S \subseteq \bigcup Z$, it follows
that $p_1 \in I$. Therefore, if $u \in S$ then $P$ is odd, and if $u \not \in S$, then $P$ is even. For every even $m \in \{1, \ldots, k\}$ let $P_m$ be the one vertex path consisting of $p_m$ (note that $p_m \in S$). Now let
$m \in \{1, \ldots, k\}$ be odd. Then $p_m=i$ for some $i \in I$. If $i=1$, let $P_1$ be the one vertex path consisting of a neighbor of $p_2$ in $H_i$. For $1<m<k$, let  $P_m$ be an even  $F$-alternating path in $H_i$ from $r_i$ to a neighbor of $p_{i+1}$ in $H_i$ (such a path exists by Lemma~\ref{hypo}). Finally, if $m=k$, let $P_m$ be an odd $F$-alternating path of $H_i$ from $r_i$ to $u$ (again, the existence of such a path is guaranteed by Lemma~\ref{hypo}).
Now the concatination $P_1 \ldots P_k$ is an odd $K,F$-alternating path in $G$, starting from a vertex outside of
$\bigcup F$ and ending in $u$. This proves that $u \in OR(K,F)$, as required.

Suppose that \eqref{33} is false, and that there exists $v \not \in Q \cup \{r_i \mid i \in I\}$ such that $v \in V \setminus OR(K,F)$. By the claim of the previous paragraph, $v \not \in S' \cup \bigcup_{i \in I'}V(H_i)$. It follows that $v \in S'' \cup \bigcup_{i \in I''}V(H_i)$, and in particular $I'' \cup S'' \neq \emptyset$.

We construct a new decomposition of $G$, that satisfies \eqref{33}.
Let $\tilde{S}=S'$, $\tilde{R}=\bigcup_{i \in I'} V(H_i)$ and $\tilde{Q}=Q \cup S'' \cup \bigcup_{i \in I''} V(H_i)$.
To show that this is a decomposition as in Theorem \ref{gallaiedmonds}, we only need to show that there are no edges between
$\tilde{Q}$ and $\tilde{R}$. Suppose such an edge $qr$ exists, where $q \in \tilde{Q}$ and $r \in \tilde{R}$.
Since $R$ is the union of component of $V(G) \setminus S$, it follows that $q \in S''$. Let $i \in I'$ be such that
$r \in V(H_i)$. Since $i \in I'$, there is a $Z$-alternating path $P$ in $\Gamma$ from a vertex outsize  $\bigcup Z$ to
$i$. Since $q \in S''$, we deduce that $q \not \in V(P)$, and in particular $q \neq Z(i)$. Since $qr \in E(G)$, it follows that $qi \in E(\Gamma)$. But now adding the edge $iq$ to $P$ results in a $Z$-alternating path in $\Gamma$ from a vertex outside  $\bigcup Z$ to $q$, contrary to the fact that $q \not \in S'$. This proves that $(\tilde{S}, \tilde{R}, \tilde{Q})$ is a decomposition as in Theorem~\ref{gallaiedmonds}, and this decomposition satisfies \eqref{33}.

To complete the proof, note that
if  $Q$ is maximal then necessarily $I'=I$ and $S'=S$, namely the decomposition at hand has the desired properties.
\end{proof}
\begin{remark}
The lemma is essentially known (see, e.g., \cite{goemans}), but is usually mentioned only implicitly.
\end{remark}

Having proved the lemma, from now on we assume that the decomposition $(Q,R,S)$ satisfies \eqref{33}.

The last step towards the proof of Lemma \ref{mainlemma} is:

\begin{lemma}\label{bridge}
An edge $e$  not belonging to $E(G)$ is
enriching if and only if it connects some $V(H_i)$ with either $V(H_k)$ for some $k \neq i$, or with $Q$.
\end{lemma}

Equivalently, $e$ is
non-enriching if and only if it satisfies one of the following conditions:

\begin{enumerate}
\item
It is incident with $S$,

\item It is contained in $Q$.

\item
It is contained in $V(H_i)$ for some  $i\in I$.

\end{enumerate}

\begin{proof}
We shall prove only the sufficiency of the condition in the lemma, which is what we need for
the proof of Lemma \ref{mainlemma}. The proof of the necessity is easier and is omitted.

Let $e=uv$ be an edge as in the assertion. One of its vertices, say $u$, belongs to $V(H_i)$ for some $i \in I$, and either

(i) $v \in Q$, or

(ii) $v \in H_k$ for some $k \in I, ~k \neq i$.

Let $P$ be an even $K-F$ alternating path from $r_i$ to $u$, contained in $V(H_i)$. If Case (ii) occurs, let $T$ be an even $K-F$ alternating path from $r_k$ to $v$ contained in $V(H_k)$.

If $i \in D$  then, in Case (i), the path $Pe$ shows that $v \in OR(K \cup \{e\}, F)$, proving that $e$ is enriching. In Case (ii), the path $Pe\overleftarrow{T}$ shows that $r_k \in OR(K \cup \{e\}, F)$, again proving that $e$ is enriching.

Thus we may assume that $i \in J$. Let $s=F(r_i)$.

Consider first Case (i). By \eqref{333} there exists an odd free-starting $K-F$ alternating  path $B$ ending at  $s$. If $V(B) \cap V(H_i)\neq \emptyset$ then, since $V(H_i)\subseteq \bigcup F$, the first edge in $B$ meeting $H_i$ is not in $F$. Let $x$ be the first vertex on $B$ belonging to $V(H_i)$, and let
let $W$ be an even $K-F$ alternating path from $r_i$ to $x$ contained in $V(H_i)$. Then the path $Bx\overleftarrow{W}$ shows that $r_i \in OR(K\cup \{e\},F)$, proving that $e$ is enriching.

So, we may assume that  $V(B) \cap V(H_i)=\emptyset$. Let $C$ be a $K-F$-alternating path witnessing the fact that $u \in ER(r_i, K,F)$ (see Lemma \ref{hypovertex}). The path $Bsr_iCuv$ then shows that $v \in \orkef$, again proving that $e$ is enriching.

Next consider Case (ii). Let $B$ be an even  free-starting alternating path terminating at $r_i$, meaning that its last edge is $sr_i$. Let  $P$ be an alternating  path witnessing the fact that $u \in ER(r_i,K,F)$, and let $W$ be a path witnessing the fact that
$v \in ER(r_k,K,F)$.  If $B$ does not meet $V(H_k)$, then  the path $Br_iP*e*\overleftarrow{W}$ shows that $r_k \in \orkef$, meaning that $e$ is enriching.
If $B$ meets $V(H_k)$, then    it must contain the edge $F(r_k)r_k$. Then the path ${W}*\overleftarrow{e}*\overleftarrow{P}$ shows that $r_i \in \orkef$, again showing that $e$ is enriching.
\end{proof}

{\em Proof of Lemma \ref{mainlemma}}  Let $A$ be a free-starting $F$-alternating path. Then $in(A)=r_d$ for some $d \in D$.
If the edge $e$ leaving $V(H_d)$ goes to $Q$ or to some $V(H_i)$, then $e$  is enriching, by Lemma \ref{bridge}. So, we may assume that the non-$V(H_i)$ vertex of $e$ is $s_j$ for some $j \in J$. The next edge on $A$ is then $s_jr_j$. By the same argument, the edge leaving $V(H_j)$ must go to $S$. Continuing this way, and noting that $ter(A) \not  \in S$, shows that one of the edges of $A$ must connect two $V(H_i)$s, or some $V(H_i)$ with $Q$. By Lemma \ref{bridge} this edge is enriching.

\section{Multicolored alternating paths and proof of Theorem \ref{main}}

Given a family (namely a multiset) $\cp$ of $F$-alternating paths, an $F$-alternating path $P$ is said to be
 {\em $\cp$-multicolored} if $E(P)\setminus F$ is a   rainbow set of the family $E(Q),~~ Q \in \cp$.

\begin{corollary}\label{maincor}
If $\cp$ is a family of augmenting $F$-alternating paths and $|\cp|>2|F|$ then there exists an augmenting $\cp$-multicolored $F$-alternating path.
\end{corollary}

\begin{proof}
By Lemma \ref{mainreach} we can construct   inductively sets of edges $K_i$, where $K_0=\emptyset$ and $K_i=K_{i-1}\cup\{e_i\}$, $e_i \in E(P_i)$, and $OR(K_{i+1},F)\supsetneqq OR(K_i, F)$. Since there are only $2|F|$ vertices in $\bigcup F$, at some point $OR(K_{i},F)$ will contain a vertex not in $\bigcup F$, meaning that there exists an augmenting $K_i-F$-alternating path $P$, which by the inductive construction of the sets $K_i$ is $\cp$-multicolored.
\end{proof}

Finally, we derive Theorem \ref{main} from Corollary \ref{maincor}. We have to show that
given $3n-2$ matchings $M_i, ~i \le 3n-2$ of size $n$ there exists a   rainbow matching of size $n$. Let $F$ be a rainbow matching of maximal size, and let $|F|=k$. We wish to show that $k=n$.  Suppose to the contrary that $k<n$. Then there are at least $2k+1$ matchings $M_i$ not represented in $F$.
Each of these generates an augmenting $F$-alternating path $P_i$, and by the corollary, there is an augmenting multicolored $F$-alternating path $P$ using edges from the paths $P_i$. None of the colors appearing in $P$ are used
in $F$, and hence $F \triangle E(P)$ is a   rainbow matching of size $k+1$, contradicting the maximality property of $k$.

\begin{remark}
In \cite{akz} it was shown that in the bipartite  case Corollary \ref{maincor} only demands $|\cp|>|F|$. In  the case of general graphs  Corollary \ref{maincor} is sharp - $2|F|$ augmenting $F$-alternating paths
do not suffice, as the following example shows.
Let $F$ be a matching $\{u_iv_i \mid i \le k-1\} \cup \{xy\}$,  let
$P_1, \ldots ,P_k$ all be  the same path $F \cup \{xu_1\}\cup
\{v_{k-1}y\}\cup \{v_iu_{i+1} \mid i \le k-2\}$ and let $P_{k+1}, \ldots
,P_{2k}$ all be equal to the same path   $F \cup \{xv_1\}\cup \{u_{k-1}y\}\cup
\{u_iv_{i+1} \mid i \le k-2\}$.
\end{remark}

{\bf Acknowledgement} We are grateful to Danny Kotlar and Tung Nguyen for useful remarks.

\end{document}